\documentclass[a4paper,10 pt]{amsart}

\usepackage{german}

\usepackage{amscd}
\usepackage[arrow, matrix, curve]{xy}
\usepackage{color}
\usepackage{relsize} 
\usepackage[bbgreekl]{mathbbol} 
\usepackage{amsfonts} 
\DeclareSymbolFontAlphabet{\mathbb}{AMSb} 
\DeclareSymbolFontAlphabet{\mathbbl}{bbold}

\newtheorem{Satz}{Satz}[section]
\newtheorem{Lemma}[Satz]{Lemma}

\theoremstyle{definition}
\newtheorem{Definition}{Definition}
\newtheorem{Remark}{Remark}
\newtheorem{Theorem}[Satz]{Theorem}

\usepackage{amssymb}
\usepackage{fullpage}
\usepackage{hyperref}

\DeclareMathOperator{\Hom}{Hom}

\DeclareMathOperator{\Aut}{Aut}
\DeclareMathOperator{\Gal}{Gal}
\DeclareMathOperator{\End}{End}

\DeclareMathOperator{\Lie}{Lie}
\DeclareMathOperator{\Spec}{Spec}
\DeclareMathOperator{\Spa}{Spa}

\DeclareMathOperator{\Spf}{Spf}
\DeclareMathOperator{\Nilp}{Nilp}

\DeclareMathOperator{\PGL}{PGL}
\DeclareMathOperator{\GL}{GL}

\DeclareMathOperator{\Displ}{Displ}

\DeclareMathOperator{\Fil}{Fil}

\DeclareMathOperator{\colim}{colim}

\DeclareMathOperator{\AlgNilp}{AlgNilp}
\DeclareMathOperator{\rk}{rk}
\DeclareMathOperator{\coker}{coker}
\DeclareMathOperator{\grad}{grad}
\DeclareMathOperator{\Set}{Set}

\begin{document}
\begin{titlepage}
\begin{center}

\title{The universal special formal $\mathcal{O}_{D}$-module for $d=2$}
\author{Sebastian Bartling}
\maketitle
\end{center}

Abstract: A new proof of an old theorem of Drinfeld concerning the representability of the moduli problem of special formal $\mathcal{O}_{D}$-modules by Deligne's $p$-adic formal model of Drinfeld's upper half-plane is given for $d=2.$ The display associated to the universal object is constructed explicitly.

\tableofcontents
\end{titlepage}
\newpage
\section{Introduction}
One of the most widely studied objects in $p$-adic geometry is the non-archimedean analogue of the upper half-plane: 
$$
\Omega^{d}_{\mathbb{Q}_{p}}=\mathbb{P}^{d-1,\text{ad}}_{\mathbb{Q}_{p}}-\mathcal{H},
$$
where $\mathcal{H}$ is the union of all $\mathbb{Q}_{p}$-rational hyperplanes inside $\mathbb{P}^{d-1}_{\mathbb{Q}_{p}}.$ One source of fascination comes from the fact that the base-change $\Omega^{d}\times_{\Spa(\mathbb{Q}_{p})}\Spa(\breve{\mathbb{Q}}_{p})$\footnote{Recall that $\breve{\mathbb{Q}}_{p}$ denotes the completion of the maximal unramified extension of $\mathbb{Q}_{p}.$} has a tower of highly non-trivial finite étale coverings, which incodes Langlands (and Jacquet-Langlands) correspondences in both $\ell$-adic (\cite{FaltingsTraceOmega}, \cite{FaltingsARelationBetween}, \cite{FarguesL'IsoEntreLesTours}, \cite{HarrisDrinfeldSpace}, \cite{HarrisTaylor}) and $p$-adic local Langlands (\cite{ColmezetalDrinfeld1} for $d=2$).
\\
These covering spaces were constructed by Drinfeld in his seminal paper \cite{DrinfeldOmega}; an article which is both famous and infamous for it's importance and being very concise. To explain how Drinfeld goes about constructing these covering spaces, one needs actually to work with an integral model of the rigid-analytic variety $\Omega^{d}_{\mathbb{Q}_{p}}.$
\\
The following construction of this integral model is due to Deligne: Let $\mathcal{B}T(\PGL_{d}(\mathbb{Q}_{p}))$ be the Bruhat-Tits building of the group $\PGL_{d}(\mathbb{Q}_{p}).$ The set of vertices (=$0$-simplices) is given by homothety classes of $\mathbb{Z}_{p}$-lattices inside of $\mathbb{Q}_{p}^{d}.$ A collection of vertices $[\Lambda_{i_{0}}],...,[\Lambda_{i_{r}}]$ forms a simplex, if there are representatives $\Lambda_{i_{0}},...,\Lambda_{i_{r}},$ such that
$$
p\Lambda_{i_{r}}\subset \Lambda_{i_{0}} \subset ... \subset \Lambda_{i_{r}}.
$$
Now one associates to a simplex $\Delta$ inside of $\mathcal{B}T(\PGL_{d}(\mathbb{Q}_{p}))$ an affine $p$-adic formal scheme as the solution of the following moduli problem: fix a set of representatives $\Lambda_{i_{0}},...,\Lambda_{i_{r}}$ for the simplex $\Delta$ and let 
$$
\widehat{U}_{\Delta}\colon \text{Nilp}^{\text{op}}\rightarrow \text{Set}
$$
be the functor which sends a scheme $S$ on which $p$ is Zariski-locally nilpotent\footnote{In the following I write $\Nilp$ for the category of these schemes.} towards equivalence classes of triples $(\mathcal{L}_{i_{k}},\alpha_{i_{k}},\Pi_{k}),$ $0\leq k \leq r$ where $\mathcal{L}_{i_{k}}$ are line bundles on $S$, $\alpha\colon \Lambda_{i_{k}}\rightarrow \mathcal{L}_{i_{k}}$ are $\mathbb{Z}_{p}$-linear morphisms and $\Pi_{k}\colon \mathcal{L}_{i_{k}}\rightarrow \mathcal{L}_{i_{k+1}}$ are $\mathcal{O}_{S}$-linear morphisms. This data is required to give rise to a commutative diagram 
$$
\xymatrix{
p\Lambda_{i_{r}} \ar[r] \ar[d]^{\alpha_{i_{r}}/p} & \Lambda_{i_{0}} \ar[r] \ar[d]^{\alpha_{i_{0}}} ... &  \Lambda_{i_{r}} \ar[d]^{\alpha_{i_{r}}} \\
\mathcal{L}_{i_{r}} \ar[r]^{\Pi_{r-1}} & \mathcal{L}_{i_{0}} \ar[r]^{\Pi_{i_{0}}} ... & \mathcal{L}_{i_{r}}.
}
$$
Furthermore, one requires that the following condition is satisfied: 
\begin{enumerate}
\item[(Deligne):] For $m\in \Lambda_{i_{k+1}}-\Lambda_{i_{k}}$ (resp. $m\in \Lambda_{i_{0}}-p\Lambda_{i_{r}}$), the section $\alpha_{i_{k+1}}(m)\in \mathcal{L}_{i_{k+1}}$ (resp. $\alpha_{i_{0}}(m)\in \mathcal{L}_{i_{0}}$) does not vanish on $S.$
\end{enumerate}
One checks that this functor in fact does not depend on the choice of the representatives which were implicit in the construction. Then it is not hard to see that $\widehat{U}_{\Delta}$ is representable by an affine $p$-adic formal scheme $\Spf(A_{\Delta}).$ For example, if $\Delta$ is a maximal simplex, then one has (non-canonically)
$$
A_{\Delta}\simeq \mathbb{Z}_{p}\lbrace X_{1},...,X_{d}, u^{-1} \rbrace/(X_{1}\cdot ...\cdot X_{d}-p),
$$
where 
$$
u=\prod_{i\in \mathbb{Z}/d\mathbb{Z},\lambda\in \mathbb{F}_{p}^{d-1}}u_{i,\lambda}
$$
with
$$
u_{i,\lambda}=1+\lambda_{1}X_{i-1}+\lambda_{2}X_{i-1}X_{i-2}+...\lambda_{d-1}X_{i-1}X_{i-2}\cdot \cdot \cdot X_{1+i-d}.
$$
If $\Delta^{\prime}$ is a subsimplex of $\Delta,$ then the natural morphism $\widehat{U}_{\Delta^{\prime}}\rightarrow \widehat{U}_{\Delta}$ is an open immersion and one defines
$$
\widehat{\Omega}^{d}_{\mathbb{Z}_{p}}=\underset{\Delta\in \mathcal{B}T(\PGL_{d}(\mathbb{Q}_{p}))} \colim \widehat{U}_{\Delta}.
$$
This is a $p$-adic semi-stable formal scheme which is indeed a formal model of $\Omega^{d}_{\mathbb{Q}_{p}}.$ 
\\
Drinfeld's key observation was that $\widehat{\Omega}^{d}_{\breve{\mathbb{Z}}_{p}}$ admits a moduli interpretation. To introduce this moduli problem one needs a bit more set-up: let $D$ be a central simple algebra over $\mathbb{Q}_{p}$ of Brauer-invariant $1/d,$ with ring of integers $\mathcal{O}_{D}$ and fixed uniformizer $\Pi\in \mathcal{O}_{D}.$ Denoting by $K_{d}/\mathbb{Q}_{p}$ the degree $d$ unramified extension with ring of integers $\mathcal{O}_{K_{d}}$  and choosing an embedding $K_{d}\hookrightarrow D,$ one gets a presentation of $\mathcal{O}_{D}$ as a non-commutative polynomial ring over $\mathcal{O}_{K_{d}}$ in the variabel $\Pi.$ 
\begin{Definition}\label{Def SFD}
Let $S\in \Nilp.$ A special formal $\mathcal{O}_{D}$-module over $S$ is the datum of a pair $(X,\iota),$ where $X$ is a $p$-divisible group over $S$ and $\iota$ is a $\mathbb{Z}_{p}$-algebra homomorphism
$$
\iota\colon \mathcal{O}_{D}\rightarrow \End(X),
$$
such that the following axiom is satisfied:
\begin{enumerate}
\item[(Drinfeld):] at all geometric points of $S$ the action of $\mathcal{O}_{K_{d}}$ splits the Lie-algebra of $X$ into a sum of $d$ distinct characters.
\end{enumerate}
\end{Definition}
Fix a special formal $\mathcal{O}_{D}$-module $\mathbb{X}$ over $k=\overline{\mathbb{F}}_{p},$ which is of height $d^{2}$ and dimension $d.$ It is an important fact that up to $\mathcal{O}_{D}$-linear isogeny there is precesily one such special formal $\mathcal{O}_{D}$-module over $k.$
\\
Then Drinfeld considered the following moduli problem, defined over $\breve{\mathbb{Z}}_{p}=W(k),$
$$
\mathcal{D}\colon \Nilp_{\breve{\mathbb{Z}}_{p}}^{\text{op}}\rightarrow \Set
$$
sending $S\in \Nilp_{\breve{\mathbb{Z}}_{p}}$ towards pairs $(X,\rho)$ up to equivalence, where $X$ is a special formal $\mathcal{O}_{D}$-module over $S$ and
$$
\rho\colon \mathbb{X}\times_{\Spec(k)}\overline{S}\dashrightarrow X\times_{S} \overline{S}
$$
is a $\mathcal{O}_{D}$-linear quasi-isogeny of height $0,$ which is defined over $\overline{S}=V(p)\subset S.$
\begin{Theorem}{(Drinfeld)}
\\
The functor $\mathcal{D}$ is representable by the formal scheme $\widehat{\Omega}^{d}_{\breve{\mathbb{Z}}_{p}}.$
\end{Theorem}
This is \cite[Theorem in Paragraph 2, p. 109]{DrinfeldOmega}.
\\
The system of finite étale coverings of $\Omega^{d}_{\breve{\mathbb{Q}}_{p}}$ is then constructed as the generic fiber of the $p^{n}$-torsion points of the universal object over $\widehat{\Omega}^{d}_{\breve{\mathbb{Z}}_{p}}.$
\\
The basic aim of this note is to give a new proof of this theorem for $d=2$ by explicitly constructing the display of the universal object.
\\
Let me outline the construction. The idea is to construct first a special formal $\mathcal{O}_{D}$-module, together with a $\mathcal{O}_{D}$-linear quasi isogeny of height $0$ towards the fixed framing object $\mathbb{X},$ over the special fiber $$\overline{\Omega}^{2}_{k}:=\widehat{\Omega}^{d}_{\breve{\mathbb{Z}}_{p}}\times_{\Spf(\breve{\mathbb{Z}}_{p})} \Spec(k).$$ This morphism
$$
\Upsilon_{0}\colon \overline{\Omega}^{2}_{k}\rightarrow \overline{\mathcal{D}}
$$
should be constructed in such a way that on $k$-rational points one automatically gets a bijection. In section \ref{Section How to find displays} below, I explain in detail how this can be done; roughly, the idea is to construct explicitly an inverse to the morphism 
$$\Phi_{k}\colon \mathcal{D}(k)\rightarrow \widehat{\Omega}^{2}_{\breve{\mathbb{Z}}_{p}}(k)$$
constructed by Drinfeld. The display of the inverse image of a point $x\in \overline{U}_{\lbrace \Lambda \rbrace}(k)$ will depend on a scalar $\lambda\in k,$ which is not $\mathbb{F}_{p}$-rational. Letting this scalar vary one obtains a display over $\overline{U}_{\lbrace \Lambda \rbrace}$ with exactly one critical index. If one has a simplex
$$
\Delta:\! p\Lambda\subset \Lambda^{\prime} \subset \Lambda
$$
one now wants to glue the two families over $\overline{U}_{\lbrace \Lambda \rbrace}$ and $\overline{U}_{\lbrace \Lambda^{\prime} \rbrace}.$ For this, one observes that the previously constructed families also extend over the locus where the variable vanishes and over this locus these two families are identified. This extension step for $d=2$ is automatic and I don't know how to do it for $d>2.$ Afterwards one glues the two previously constructed families along the closed point of intersection to obtain a family of displays over $\overline{U}_{\Delta}.$
At this point one has constructed the desired morphism
$$
\Upsilon_{0}\colon \overline{\Omega}^{2}_{k}\rightarrow \overline{\mathcal{D}}.
$$
It was insured that $\Upsilon_{0}(k)$ is a bijection. Afterwards one uses Grothendieck-Messing deformation theory to lift the previously constructed family over $\overline{\Omega}^{2}_{k}$ to all $p$-adic thickenings. To conclude that this morphism
$$
\Upsilon\colon \widehat{\Omega}^{2}_{\breve{\mathbb{Z}}_{p}}\rightarrow \mathcal{D}
$$
is an isomorphism one argues as Drinfeld does by checking that it is enough to show that $\Upsilon_{0}$ is an isomorphism; for this one just has to convince oneself that this morphism is étale which is handled again by using Grothendieck-Messing deformation theory.
\subsection{Overview}
Here is a quick overview of the sections below: In the first section I set up some notations. In the second section I explain how to reformulate the moduli problem $\mathcal{D}$ completly in terms of displays; furthermore, basics of the deformation theory of special $\mathcal{O}_{D}$-displays are discussed. The third section explains in depth how to find the right displays over the special fiber. Afterwards, in section four, the family of special $\mathcal{O}_{D}$-displays over $\widehat{\Omega}^{2}_{\breve{\mathbb{Z}}_{p}}$ is constructed. After some reductions in section five, in the final section it is checked that the natural transformation $\Upsilon$ is really an isomorphism.
\subsection{Acknowledgments}
This project was given to me by Thomas Zink many years ago; most of the ideas here are clearly due to him. I want to thank him for the significant role he played in my mathematical upbringing. Furthermore, I thank Johannes Anschütz for giving me the possibility to talk about this project in Munich.
\\
This is part of the author's PhD thesis under the direction of Laurent Fargues. The author received financial support from the ERC Advanced Grant 742608 GeoLocLang.
\subsection{Notations}
I will work over $K=\mathbb{Q}_{p},$ $d\geq 2$ is some integer, which will mostly be just $d=2,$ $\overline{\mathbb{Q}}_{p}$ is a fixed algebraic closure, inducing an algebraic closure $k=\overline{\mathbb{F}}_{p}$ of $\mathbb{F}_{p},$ which I use to built the maximal unramified extension $W(\overline{\mathbb{F}}_{p})[1/p]=\breve{\mathbb{Q}}_{p},$ contained inside of $\overline{\mathbb{Q}}_{p}.$ Sometimes I will write $W=W(\overline{\mathbb{F}}_{p})$ and $\sigma\in \Aut(W)$ always denotes the Wittvector-Frobenius on $W.$ Let $\mathbb{F}_{p^{d}}$ be the degree $d$ extension of $\mathbb{F}_{p},$ contained inside of $\overline{\mathbb{F}}_{p}$ and $K_{d}=W(\mathbb{F}_{	p^{d}})[1/p]$ denotes the unramified extension of degree $d$ of $\mathbb{Q}_{p},$ with Frobenius-automorphism denoted by $\tau_{d}\in \Gal(K_{d}/\mathbb{Q}_{p}),$ and which comes by the set-up with a canonical embedding $$\psi_{0}\colon K_{d}\hookrightarrow \breve{\mathbb{Q}}_{p} \hookrightarrow \overline{\mathbb{Q}}_{p},$$ $D$ is a central simple $\mathbb{Q}_{p}$-algebra of Brauer-invariant $1/d,$ with ring of integers $\mathcal{O}_{D}$ and I fix throughout a uniformizer $\Pi\in \mathcal{O}_{D},$ which then satisfies $\Pi^{d}=p.$ Choosing an embedding $\mathcal{O}_{K_{d}}\subseteq \mathcal{O}_{D},$ one gets the presentation $\mathcal{O}_{D}=\mathcal{O}_{K_{d}}[\Pi]$ as the non-commutative polynomial ring in $\Pi$ with the conditions that $\Pi^{d}=p$ and $\Pi\cdot x=\tau_{d}(x)\cdot \Pi.$ A $p$-adic ring will always mean a $p$-adically complete and seperated ring. If $R$ is some $\mathbb{Z}_{p}$-algebra, I write $\Nilp_{R}$ for the category of $R$-schemes on which Zariski-locally $p$ is nilpotent.
\newpage
\section{Reformulating the moduli problem}
The aim of this section is to reformulate Drinfeld's moduli problem of special formal $\mathcal{O}_{D}$-modules using Zink's theory of displays.
\\
Let $R\in \Nilp$ and consider a special formal $\mathcal{O}_{D}$-module over $\Spec(R)$ as in Definition \ref{Def SFD}. Let me start with the following comments concerning this concept:
\begin{Remark}\label{Remark zu SFD}
\begin{enumerate}
\item[(i):]
In general, the height of a special formal $\mathcal{O}_{D}$-module over $S$ is locally constant and given by a multiple of $d^{2}:$ To see this, let $S$ be connected, so that the height is constant. Let $\bar{x}\colon\Spec(k)\rightarrow S$ be a geometric point and since $\text{ht}(X)=\text{ht}(\bar{x}^{*}(X)),$ one has to determine the height of the pullback $\bar{x}^{*}(X).$ Here one can use Dieudonné-theory to replace $\bar{x}^{*}(X)$ by a tuple $(M,V,F,\Pi,\mathbb{Z}/d\mathbb{Z}-\text{grad})$ as in Lemma \ref{Beschreibung von SFD via Displays}. Since $\Pi\colon M_{i}\rightarrow M_{i+1}$ is an isomorphism after inverting $p,$ one sees that $\rk_{W}(M_{i})=\rk_{W}(M_{i+1})$ for all $i.$ Let $r=\rk_{W}(M_{0})$ be this common value of the rank. By the special-condition, it follows that $\lg_{W}(\coker(\Pi_{i}))=\lg_{W}(\coker(\Pi_{i+1}))$ for all $i.$ Then one gets, using that $\Pi^{d}=\cdot p$ and that $\Pi$ is thus injective,
\begin{align*}
r = & \dim_{k}[M_{0}:\Pi^{d}M_{0}] \\
 = & \sum_{i=0}^{d-1}\lg_{W}(\coker(\Pi_{i})) \\
 = & d\cdot \lg_{W}(\coker(\Pi_{0})).
\end{align*}
In total, one sees that $\rk_{W}(M)=\sum_{i=0}^{d-1}\rk_{W}M_{i}=d\cdot \rk_{W}(M_{0})=d^{2}\cdot \lg_{W}(\coker(\Pi_{0})).$
\item[(ii):] I recall the important fact that over $k$ there is up to $\mathcal{O}_{D}$-linear isogeny only one special formal $\mathcal{O}_{D}$-module of height $d^{2}$ and dimension $d$ (c.f. \cite[Lemma 3.60]{RZ}). This is the key reason why the formal scheme representing the Drinfeld moduli problem is $p$-adic and the Drinfeld case is basically the only case of RZ-spaces where such a phenomen can be observed.
\item[(iii):] A special formal $\mathcal{O}_{D}$-module of height $d^{2}$ and dimension $d$  is automatically a formal $p$-divisible group, i.e. at all geometric points of $S$ it has no étale part. Again, one may check this at a geometric point $\bar{x}\colon \Spec(k)\rightarrow S$ and use Dieudonné-theory there to replace $\bar{x}^{*}(X)$ by a tuple $(M,F,V,\Pi,\mathbb{Z}/d\mathbb{Z}-\text{grad})$ as before. Then it suffices to see that $V$ is topologically nilpotent on $M,$ i.e. it suffices that $V$ is nilpotent on $M/pM=M/\Pi^{d}M.$ Now $M/\Pi^{d}M$ has a descending filtration given by $\Pi^{i}M/\Pi^{d}M,$ $i=1,...,d-1,$ whose graded pieces are always given by $M/\Pi M.$ It follows inductively then that it suffices to show that $V$ acts nilpotently on $M/\Pi M.$ But this follows from the fact that there exists at least one critical index: for all $i,$ the morphism $$\bar{V}\colon \coker(\Pi_{i})\rightarrow \coker(\Pi_{i+1})$$ is either an isomorphism or the zero morphism (here I use the condition on the height: it implies by the previous remark (i) above that $\lg_{W}(M_{i+1}/\Pi M_{i})=1$ ); if it were always an isomorphism, then also $\bar{V}\colon M/\Pi M\rightarrow M/ \Pi M$ would be an isomorphism, so that also $\bar{\Pi}\colon \Lie(X)\rightarrow \Lie(X)$ would be an isomorphism, which is wrong since one finds a critical index. Here I recall that an index $i\in \mathbb{Z}/d\mathbb{Z}$ is called critical if the induced morphism
$$
\Pi\colon \Lie(X)_{i}\rightarrow \Lie(X)_{i+1}
$$
is the zero morphism.
\item[(iv):] One can formulate condition (a) equivalently as follows: étale locally on $S$ (here it suffices to take the étale covering $S\times_{\Spec(\mathbb{Z}_{p})}\Spec(\mathcal{O}_{K_{d}}))\rightarrow S$), if one looks at the eigenspace decomposition
$$
\Lie(X)=\bigoplus_{\psi\colon \mathcal{O}_{K_{d}}\hookrightarrow \mathcal{O}_{S}} (\Lie(X))_{\psi},
$$ 
where $\mathcal{O}_{K_{d}}\subset \mathcal{O}_{D}$ acts on $\Lie(X)_{\psi}$ via $\psi,$ then all $(\Lie(X))_{\psi}$ are line-bundles over $S.$
\end{enumerate}
\end{Remark}
Part (iii) of the previous remark opens the way to describe a special formal $\mathcal{O}_{D}$-module via Zink's theory of displays (\cite{ZinkDisplay}); amplified by a theorem of Lau (\cite{LauInventiones}), which says that Zink's classification of formal $p$-divisible groups by nilpotent displays extends to all $p$-adic rings. Using this, the next lemma is rather formal and well-known:
\begin{Lemma}\label{Beschreibung von SFD via Displays}
Let $\Spec(R)\in \Nilp_{\breve{\mathbb{Z}}_{p}}.$Then there is an equivalence of categories between special formal $\mathcal{O}_{D}$-modules of height $d^{2}$ and dimension $d$ over $\Spec(R)$ and the category of triples $(\mathcal{P},\Pi,\mathbb{Z}/d\mathbb{Z}-\grad),$ where $\mathcal{P}$ is a display over $R$, which is locally free on $R$ of rank $d^{2},$ $\Pi\in \End_{\Displ}(\mathcal{P})$ and the datum of a $\mathbb{Z}/d\mathbb{Z}$-grading $P=\bigoplus_{i\in \mathbb{Z}/d\mathbb{Z}} P_{i},$ $Q=\bigoplus_{i\in \mathbb{Z}/d\mathbb{Z}} Q_{i},$ such that $Q_{i}=P_{i}\cap Q,$ $\deg(F)=\deg(\dot{F})=+1,$ $\deg(\Pi)=+1, $ $\Pi^{d}=\cdot p,$ all $P_{i}/Q_{i}$ are locally free of rank $1$ over $R.$\footnote{Note that such a triple $(\mathcal{P},\Pi,\mathbb{Z}/d\mathbb{Z}-\grad)$ is automatically a nilpotent display: at all geometric points of $\Spec(R),$ the display identifies with a Dieudonné-module and one can argue as in Remark \ref{Remark zu SFD} (iii).}
\end{Lemma}
\begin{proof}
Recall that I fixed one embedding $\psi_{0}\colon K_{d}\hookrightarrow \overline{\mathbb{Q}}_{p},$ so that one has $$\Hom(K_{d},\overline{\mathbb{Q}}_{p})\simeq \lbrace \psi_{0}\circ \tau_{d}^{-i}\colon i\in \mathbb{Z}/d\mathbb{Z} \rbrace.$$ Let me denote $\psi_{i}=\psi_{0}\circ \tau_{d}^{-i}.$ Using the Cartier-morphism $\triangle\colon \breve{\mathbb{Z}}_{p}\rightarrow W(\breve{\mathbb{Z}}_{p}),$ one can then give $W(R)$ the structure of a $\mathcal{O}_{K_{d}}$-algebra via $\psi_{i},$ i.e. one considers the composition
$$
\xymatrix{
\mathcal{O}_{K_{d}} \ar[r]^{\psi_{i}} & \breve{\mathbb{Z}}_{p} \ar[r] & W(\breve{\mathbb{Z}}_{p}) \ar[r] & W(R),
}
$$
which I will again denote by $\psi_{i}.$ Since $R$ is a $\breve{\mathbb{Z}}_{p}$-algebra, for a given special formal $\mathcal{O}_{D}$-module $(X,\iota)$ the $\mathcal{O}_{S}$-module $\Lie(X)$ already splits under the action of $\mathcal{O}_{K_{d}}$ into eigenspaces which are line-bundles (i.e. one does not need to pass to any étale covering of $\Spec(R)$). Now let $\mathcal{P}=(P,Q,F,\dot{F})$ be the nilpotent display over $R$ associated to the $p$-divisible group $X$ over $\Spec(R).$ By the full faithfullness part of the equivalence of categories between nilpotent displays and formal $p$-divisible groups, one gets a $\mathbb{Z}_{p}$-algebra homomorphism
$$
\iota\colon \mathcal{O}_{D}\rightarrow \End_{\Displ}(\mathcal{P}).
$$
The $\mathbb{Z}/d\mathbb{Z}$-grading is then simply defined as follows: 
$$
P_{i}=\lbrace x\in P\colon \iota(a)(x)=\psi_{i}(a)\cdot x \text{ for all }a\in \mathcal{O}_{K_{d}} \rbrace,
$$
and
$$
Q_{i}=\lbrace y\in Q\colon \iota(a)(y)=\psi_{i}(a)\cdot y \text{ for all }a\in \mathcal{O}_{K_{d}} \rbrace.
$$
Then one has that $P=\bigoplus_{i\in \mathbb{Z}/d\mathbb{Z}} P_{i},$ because it is the decomposition induced on $P$ by the isomorphism $W(R)\otimes_{\mathbb{Z}_{p}} \mathcal{O}_{K_{d}}\simeq \prod_{i\in \mathbb{Z}/d\mathbb{Z}} (W(R))_{\psi_{i}},$ and similiarly for $Q$. The datum of $\Pi\in \End_{\Displ}(\mathcal{P})$ is then simply $\iota(\Pi)$ and the condition $\deg(F)=\deg(\dot{F})=+1$ follows, since these are $F$-linear morphisms, while $\deg(\Pi)=+1$ follows from the equation $\Pi \cdot a=\tau_{d}(a)\cdot \Pi$ inside of $\mathcal{O}_{D}.$ Now, recall that the finite projective $R$-module $P/Q$ identifies with the $R$-module $\Lie(X)$ and $P_{i}/Q_{i}$ corresponds then under this identification to the $\psi_{i}$-eigenspace, so that it has to be finite locally free of rank $1$; it is easy to see that one may reverse this whole process and get from the triple $(\mathcal{P},\Pi,\mathbb{Z}/d\mathbb{Z}-\text{grad})$ a special formal $\mathcal{O}_{D}$-module $(X,\iota)$ over $\Spec(R)$ and this finishes the proof.
\end{proof}
One can furthermore translate quasi-isogenies as follows:
\begin{Lemma}
Let $R\in \Nilp_{\breve{\mathbb{Z}}_{p}}$ and $X,Y$ be special formal $\mathcal{O}_{D}$-modules over $\Spec(R)$ and let $\mathcal{P}(X), \mathcal{P}(Y)$ be the corresponding special $\mathcal{O}_{D}$-displays. Then the datum of a $\mathcal{O}_{D}$-linear quasi-isogeny
$$
\rho\colon X \dashrightarrow Y
$$
is equivalent to the datum of a $\mathbb{Z}/d\mathbb{Z}$-graded, $\Pi$-compatible isomorphism of isodisplays
$$
\rho\colon \mathcal{P}(X)[1/p]\simeq \mathcal{P}(Y)[1/p].
$$
\end{Lemma}
For the definition of isodisplays, please see \cite[Definition 61]{ZinkDisplay} and then \textit{loc.cit.} Example 63 how to associate to a display an isodisplay.
\begin{proof}
This follows because a quasi-isogeny between formal $p$-divisible groups translates into an isomorphism of the isodisplays corresponding to the nilpotent displays of the respective formal $p$-divisible groups.
\end{proof}
In total, this allows one to reformulate the moduli problem of special formal $\mathcal{O}_{D}$-modules up to isogeny completely in terms of displays: Let $(N,\varphi_{N},\iota)$ be the $\mathcal{O}_{D}$-isodisplay associated to the special $\mathcal{O}_{D}$-display of $\mathbb{X}.$ The functor 
$$
\mathcal{D}\colon \Nilp_{W}^{op}\rightarrow \Set
$$
introduced in the introduction is isomorphic to the functor that sends a $p$-nilpotent $W$-algebra $R$ to equivalence classes of tuples $\lbrace (\mathcal{P},\Pi,\mathbb{Z}/d\mathbb{Z}-\text{grad},\rho) \rbrace/ \sim$ where $(\mathcal{P},\Pi,\mathbb{Z}/d\mathbb{Z}-\text{grad})$ is a special $\mathcal{O}_{D}$-display over $R$ and 
$$
\rho\colon N\otimes_{W_{\mathbb{Q}}(k)}W_{\mathbb{Q}}(R/pR)\simeq \mathcal{P}_{R/pR}[1/p],
$$
is an isomorphism of special $\mathcal{O}_{D}$-Isodisplays, which is of height $0.$ Two pairs $(\mathcal{P},\Pi,\mathbb{Z}/d\mathbb{Z}-\text{grad},\rho)$ and $(\mathcal{P}^{\prime},\Pi^{\prime},\mathbb{Z}/d\mathbb{Z}-\text{grad},\rho^{\prime})$ are equivalent, if the $\mathcal{O}_{D}$-linear quasi-isogeny 
$$
\rho^{\prime}\circ \rho^{-1}\colon \mathcal{P}^{\prime}_{R/pR}\dashrightarrow \mathcal{P}_{R/pR}
$$
lifts to an isomorphism of special $\mathcal{O}_{D}$-displays $(\mathcal{P},\Pi,\mathbb{Z}/d\mathbb{Z}-\text{grad})\simeq (\mathcal{P}^{\prime},\Pi^{\prime},\mathbb{Z}/d\mathbb{Z}-\text{grad}).$
\\
I next discuss a bit the deformation theory of special formal $\mathcal{O}_{D}$-modules from the point of view of displays. Let me consider a pd-thickening of $p$-nilpotent $\breve{\mathbb{Z}}_{p}$-algebras $S\twoheadrightarrow R$ with kernel denoted by $\mathfrak{a}.$ One requires that the pd-structure $\gamma$ on $\mathfrak{a}$ is compatible with the natural one on $p\mathbb{Z}_{p}.$ Recall that one then has Zink's relative Witt-frame $\mathcal{W}(S/R),$ which comes with a strict nil-crystalline frame-morphism
$$
\epsilon_{S/R}\colon \mathcal{W}(S/R)\rightarrow \mathcal{W}(R).
$$
This frame is defined as follows: let 
$$
\Fil(W(S)):=\ker(W(S)\rightarrow R)=W(\mathfrak{a})+I_{S}.
$$
Zink's logarithmic ghost components (c.f. \cite[section 1.4.]{ZinkDisplay}) define an isomorphism
$$
\log\colon W(\mathfrak{a})\simeq \prod_{i=0}^{\infty}\mathfrak{a}.
$$
This gives the splitting $W(\mathfrak{a})=\mathfrak{a}\oplus I(\mathfrak{a}),$ where $I(\mathfrak{a})=V(W(\mathfrak{a}))$ and $\Fil(W(S))=\mathfrak{a}\oplus I(S).$ Using this one can define (c.f. \cite[Lemma 38]{ZinkDisplay})
$$
\dot{F}\colon \Fil(W(S))\rightarrow W(S)
$$
via $\dot{F}\!\mid_{I(S)}=V^{-1}$ and $\dot{F}\!\mid_{\mathfrak{a}}=0.$ In total, one has the frame with constant $p$
$$
\mathcal{W}(S/R)=(W(S),\Fil(W(S)),R, F,\dot{F}).
$$
The main theorem of the deformation theory of displays (\cite[Thm. 44 ]{ZinkDisplay}) says that 
$$
\epsilon_{S/R}\colon \mathcal{W}(S/R)\rightarrow \mathcal{W}(R)
$$
induces an equivalence of categories between nilpotent displays over $R$ and nilpotent windows over the frame $\mathcal{W}(S/R).$
\begin{Lemma}\label{Lemma zu relativen SFD Windows}
The category of special $\mathcal{O}_{D}$-displays of height $d^{2}$ and dimension $d$ over $R$ is equivalent to the category of triples $(\widetilde{\mathcal{P}},\Pi,\mathbb{Z}/d\mathbb{Z}-\grad),$ where $\widetilde{\mathcal{P}}$ is a nilpotent $\mathcal{W}(S/R)$-window of height $d^{2}$, $\Pi\in \End_{\mathcal{W}(S/R)}(\widetilde{\mathcal{P}}),$ such that $\Pi^{d}=\cdot p,$ and one has a $\mathbb{Z}/d\mathbb{Z}$-grading $\widetilde{P}=\bigoplus_{i\in \mathbb{Z}/d\mathbb{Z}}\tilde{P}_{i},$ $\widetilde{Q}=\bigoplus_{i\in \mathbb{Z}/d\mathbb{Z}}\tilde{Q}_{i},$ $\tilde{Q}_{i}=\tilde{Q}\cap \tilde{P}_{i},$ $\deg(\dot{F})=\deg(F)=+1,$ $\deg(\Pi)=+1$ and for the finite projective $R$-module $\tilde{P}/\tilde{Q},$ one requires that all graded parts $\tilde{P}_{i}/\tilde{Q}_{i}$ are finite locally free $R$-modules of rank $1.$
\end{Lemma}
\begin{proof}
This follows formally by chasing the additional structures through the equivalence of categories induced by the frame morphism $\epsilon_{S/R}.$
\end{proof}
The statement I will actually need for the construction of the universal special $\mathcal{O}_{D}$-display is the following:
\begin{Lemma}\label{Deformationstheorie von SFD-displays}
Let $(\mathcal{P},\Pi, \mathbb{Z}/d\mathbb{Z}-\grad)$ be a special $\mathcal{O}_{D}$-display over $R.$ Denote by $(\widetilde{\mathcal{P}},\Pi, \mathbb{Z}/d\mathbb{Z}-\grad)$ the uniquely determined lift, as provided by the previously lemma. Then to give a lift of $(\mathcal{P},\Pi, \mathbb{Z}/d\mathbb{Z}-\grad)$ towards a special $\mathcal{O}_{D}$-display over $S$ is the same as giving oneself a colletion of surjections of finite projective $S$-modules
$$
\tilde{P}_{i}/I(S)\tilde{P}_{i}\twoheadrightarrow L_{i},
$$
for $i\in \mathbb{Z}/d\mathbb{Z},$  and morphisms $L_{i}\rightarrow L_{i+1},$ such that
\begin{enumerate}
\item[(a):] For all $i\in \mathbb{Z}/d\mathbb{Z}$ one has a commutative diagram
$$
\xymatrix{
\tilde{P}_{i} \ar[r]^{\Pi} \ar[d] & \tilde{P}_{i+1} \ar[d] \\
L_{i} \ar[r] & L_{i+1},
}
$$,
\item[(b):] this is a lift of the Hodge-filtration $P_{i}/I(R)P_{i}\twoheadrightarrow P_{i}/Q_{i}$ (so that all $L_{i}$ are automatically line bundles on $S$) and of the corresponding commutative diagram
$$
\xymatrix{
P_{i} \ar[r]^{\Pi} \ar[d] & P_{i+1} \ar[d] \\
P_{i}/Q_{i} \ar[r] & P_{i+1}/Q_{i+1}.
}
$$
\end{enumerate}
\end{Lemma}
\begin{proof}
This is a formal consequence of the fact that lifts of nilpotent displays along pd-thickenings correspond to lifts of the Hodge-filtration \cite[Thm. 48]{ZinkDisplay}. The additional assumptions one puts on this lift of the Hodge-filtration take care of the structure of a $\mathcal{O}_{D}$-display.
\end{proof}
\section{How to find the displays over the special fiber}\label{Section How to find displays}
In this section I want to explain how one can find the right displays over the special fiber; at least over the opens where exactly (!) one (resp. a certain number of) indice(s) is (are) critical. Let me first explain what I mean by the right displays: The idea is to first construct a morphism
$$
\Upsilon_{0}\colon \widehat{\Omega}^{2}_{\breve{\mathbb{Z}}_{p}}\times_{\Spf(\breve{\mathbb{Z}}_{p})}\Spec(k)\rightarrow \mathcal{D}\times_{\Spf(\breve{\mathbb{Z}}_{p})}\Spec(k),
$$
which, since $\widehat{\Omega}^{2}_{\breve{\mathbb{Z}}_{p}}\times_{\Spf(\breve{\mathbb{Z}}_{p})}\Spec(k)$ is constructed from glueing all the $\overline{U}_{\Delta}=\Spec(\overline{A}_{\Delta}),$ just amounts to giving oneself compatible special $\mathcal{O}_{D}$-displays $(\mathcal{P}_{\Delta},\Pi,\mathbb{Z}/d\mathbb{Z}-\grad)$ over $ \overline{U}_{\Delta}=\Spec(\overline{A}_{\Delta}),$ together with a $\mathcal{O}_{D}$-linear quasi-isogeny $\rho_{\Delta}$ to the chosen standard $\mathcal{O}_{D}$-isodisplay over $k.$ Compatible means that whenever one has a sub-simplex $\Delta^{\prime}\subseteq \Delta$ and $j\colon \overline{U}_{\Delta^{\prime}}\hookrightarrow \overline{U}_{\Delta}$ denotes the canonical open immersion, then
$$
j^{*}(\mathcal{P}_{\Delta})\simeq \mathcal{P}_{\Delta^{\prime}}.
$$
Then these displays should be constructed in such a way, that $\Upsilon_{0}$ induces automatically a bijection on $\Spec(k)$-rational points. To achieve this, the idea is simply to start with the morphism $$\Phi_{k}\colon \mathcal{D}(k)\rightarrow \widehat{\Omega}^{2}_{\breve{\mathbb{Z}}_{p}}(k)$$ constructed by Drinfeld and then use its inverse. This will then naturally lead to the wanted displays.
\\
Let me first recall this map $\Phi_{k}:$ I denoted by $\mathbb{X}$ the chosen special formal $\mathcal{O}_{D}$-module of height $d^{2}$ and dimension $d$ over $\Spec(k).$ Denote by $(M(\mathbb{X}),V,\Pi,\mathbb{Z}/d\mathbb{Z}-\text{grad})$ the corresponding Dieudonné-module and by $(N, V, \Pi, \mathbb{Z}/d\mathbb{Z}-\text{grad})$ the associated  $\mathcal{O}_{D}$-Isocrystal over $k.$ Then one will build the $p$-adic formal scheme $\widehat{\Omega}^{d}_{\mathbb{Z}_{p}}$ with respect to the Bruhat-Tits building of the group $\PGL(N_{0}^{V^{-1}\Pi}).$ 
\\
Now I will choose a trivialization of the framing isocrystal. Denote by $e_{1},...,e_{d}$ the standardbasis of $W_{\mathbb{Q}}(k)^{d}.$ Then let $N_{i}=W_{\mathbb{Q}}(k)^{d}$ for all $i\in \lbrace 0,...,d-1\rbrace$ and for $i\neq d-1$ let $\Pi\colon N_{i}\rightarrow N_{i+1}$ be the identity while finally $\Pi\colon N_{d-1}\rightarrow N_{0}$ is the homothety given by multiplication with $p.$ Furthermore, for $i\neq d-1,$ $V\colon N_{i}\rightarrow N_{i+1}$ is given by $\sigma^{-1}$ and $V\colon N_{d-1}\rightarrow N_{0}$ is given by $\sigma^{-1}\cdot p.$ Having all this set up, one can finally recall the construction of $\Phi_{k}:$ Let $(M,V,\Pi,\mathbb{Z}/d\mathbb{Z}-\text{grad})$ be a special $\mathcal{O}_{D}$-Dieudonné-module over $k$ and one is given an inclusion $\rho\colon M\hookrightarrow N,$ which is compatible with all additional structure; this determines the operators $V$ and $\Pi$ on the Dieudonné-module. Now let $\lbrace i_{1},...,i_{r}\rbrace $ be the collection of critical indices for $M.$ Then $V^{-1}\Pi\colon M_{i_{k}}\rightarrow M_{i_{k}}$ is well-defined and by an easy Eigenbasis-lemma (here one uses e.g. that $k$ is algebraically closed), one sees that $\Lambda_{i_{k}}=M_{i_{k}}^{V^{-1}\Pi}$ is a free $\mathbb{Z}_{p}$-module of rank $d,$ which is via the quasi-isogeny canonically identified with a lattice inside of $\mathbb{Q}_{p}^{d}=N_{0}^{V^{-1}\Pi}.$ In total, one sees that the collection of these lattices $\lbrace \Lambda_{i_{1}},...,\Lambda_{i_{r}} \rbrace$ inside of $N_{0}^{V^{-1}\Pi}$ gives a simplex $\Delta$ in the Bruhat-Tits building of $\PGL(N_{0}^{V^{-1}\Pi}).$ One can now consider the diagram
$$
\xymatrix{
\Lambda_{i_{r}} \ar[d] \ar[r] & \Lambda_{i_{1}} \ar[d] \ar[r] & ...& \ar[r]& \Lambda_{i_{r}} \ar[d] \\
M_{i_{r}}/VM_{i_{r}-1} \ar[r]^{0} & M_{i_{1}}/VM_{i_{1}-1} \ar[r]^{0} &...&\ar[r]^{0} & M_{i_{r}}/VM_{i_{r}-1}.
}
$$
I claim that this defines indeed a point of $\widehat{U}_{\Delta}(\Spec(k)).$
\\
To verify this, one has to see that the condition (Deligne) is satisfied, i.e. argueing by contradiction one has to see that if $m\in M_{i_{k}}^{V^{-1}\Pi=1},$ such that $m$ vanishes under the morphism $M_{i_{k}}^{V^{-1}\Pi=1}\rightarrow M_{i_{k}}/VM_{i_{k}-1},$ then it is the image under $\Pi$ of some element $n\in M_{i_{k-1}}^{V^{-1}\Pi=1}.$ First, recall that the index $i_{k}$ being critical means that the morphism induced by $\Pi$ on Lie-algebras, $\Lie_{i_{k}}\rightarrow \Lie_{i_{k}+1}$ is constant zero. This implies that $\Lie_{i_{k}}$ is isomorphic to $\ker(\bar{V}\colon M_{i_{k}}/\Pi M_{i_{k}-1}\rightarrow M_{i_{k}+1}/\Pi M_{i_{k}})$ by the snake lemma: here one looks at the following commutative diagram of short exact sequences
$$
\xymatrix{
 & 0 \ar[d] & 0 \ar[d] & & \\
0 \ar[r] & M_{i} \ar[r]^{\Pi} \ar[d]^{V} & M_{i+1} \ar[r] \ar[d]^{V} & M_{i+1}/\Pi M_{i} \ar[d]^{\bar{V}} \ar[r] & 0 \\
0 \ar[r] & M_{i+1} \ar[r]^{\Pi} \ar[d] & M_{i+2} \ar[r] \ar[d] & M_{i+2}/\Pi M_{i+1} \ar[r] & 0 \\
& \Lie_{i+1} \ar[r]^{\bar{\Pi}} & \Lie_{i+2}. & &
}
$$
Now consider $m\in M_{i_{k}}^{V^{-1}\Pi=1}.$ Since $V(m)=\Pi(m),$ it follows that $$m\in \ker(\bar{V}\colon M_{i_{k}}/\Pi M_{i_{k}-1}\rightarrow M_{i_{k}+1}/\Pi M_{i_{k}}).$$ If the image of $m$ is zero under
$$
M_{i_{k}}^{V^{-1}\Pi=1}\rightarrow M_{i_{k}}/VM_{i_{k}-1},
$$
then it follows from the previously explained isomorphism
$$
\Lie_{i_{k}}\simeq \ker(\bar{V}\colon M_{i_{k}}/\Pi M_{i_{k}-1}\rightarrow M_{i_{k}+1}/\Pi M_{i_{k}})
$$
that also the image of $m$ in $M_{i_{k}}/\Pi M_{i_{k}-1}$ vanishes. This means that one finds a uniquely determined (since $\Pi$ is injective) $m^{\prime}\in M_{i_{k}-1},$ such that
$$
\Pi(m^{\prime})=m.
$$
The next step is to construct a uniquely determined element $m^{\prime \prime}\in M_{i_{k}-2},$ such that
$$
\Pi(m^{\prime \prime})=m^{\prime}.
$$
Here the assumption that $i_{k-1}$ is not critical is crucial: by definition this means that
$$
\overline{\Pi}\colon \Lie_{i_{k}-1}\simeq \Lie_{i_{k}}.
$$
Using the snake lemma and again the commutative diagram above, this implies
$$
\overline{V}\colon M_{i_{k}-1}/\Pi M_{i_{k}-2}\simeq M_{i_{k}}/\Pi M_{i_{k}-1}.
$$
To see that the image $\overline{m^{\prime}}$ of $m^{\prime}$ in $M_{i_{k}-1}/\Pi M_{i_{k}-2}$ vanishes, it therefore suffices to see that 
$$
\overline{V}(\overline{m^{\prime}})=0.
$$
This would be the case, if $V(m^{\prime})=\Pi(m^{\prime}).$ Since $\Pi$ is injective, this equality can be checked after application of $\Pi,$ where it then follows from the assumption that $m\in M_{i_{k}}^{V^{-1}\Pi=1}.$ Continuing in this fashion, one sees that the condition (Deligne) is really satisfied.
\\
Now I want to construct the inverse of the previously constructed morphism $\Phi_{k}.$ For this, I have to explain how to reconstruct from a given $\Spec(k)$-valued point of $\widehat{U}_{\Delta}$ a special $\mathcal{O}_{D}$-Dieudonné-module, which is included into $(N,V,\Pi,\mathbb{Z}/d\mathbb{Z}-\text{grad}).$ Here $\Delta$ is a simplex $\lbrace [\Lambda_{i_{1}}],...,[\Lambda_{i_{r}}]\rbrace$ and in the following I will choose representatives sitting in a chain
$$
p\Lambda_{i_{r}}\subset \Lambda_{i_{1}}\subset ... \subset \Lambda_{i_{r}}.
$$
For the special formal $\mathcal{O}_{D}$-Dieudonné-module I am about to construct exactly the indices $i_{1},...,i_{r}$ will be critical and for a critical index $i_{k}$ one can set $M_{i_{k}}=\Lambda_{i_{k}}\otimes_{\mathbb{Z}_{p}} W(k),$ so that one has to find the right graded pieces $M_{i_{k}-1},...,M_{i_{k-1}+1},$ sitting inside $N_{i_{k}-1},...,N_{i_{k-1}+1}$ such that the operators $V$ and $\Pi$ of $(N,V,\mathbb{Z}/d\mathbb{Z}-\text{grad})$ will induce $V,\Pi\colon M_{i_{k}-l-1}\rightarrow M_{i_{k}-l}.$ Now for $i_{k-1}$ one looks at the arrow $M_{i_{k}}\rightarrow L_{i_{k}},$ which is given (by the point in the Omega). If one would already know $M_{i_{k}-1},$ one would have $L_{i_{k}}=M_{i_{k}}/VM_{i_{k}-1},$ so that one is led to consider
$$
I_{i_{k}-1}=\ker(M_{i_{k}}\rightarrow L_{i_{k}}).
$$
This is a free $W$-module of rank $d.$ I define then $M_{i_{k}-1}=W\otimes_{F,W}I_{i_{k}-1},$ still a free $W$-module of rank $d,$ which I equip with
$$
V\colon M_{i_{k}-1}\rightarrow I_{i_{k}-1}\subseteq M_{i_{k}},
$$
given by $w\otimes m\mapsto \sigma^{-1}(w)\cdot m.$ To define the operator $\Pi,$ recall that I consider $M_{i_{k}}$ as included inside of the isocrystal $N,$ on which I have the operator $U=V^{-1}\Pi$ given and one knows that $U(M_{i_{k}})\subseteq M_{i_{k}}.$ It therefore makes sense to define
$$
\Pi=U\circ V\colon M_{i_{k}-1}\rightarrow M_{i_{k}}.
$$
Next I will consider the index $i_{k}-2:$ To motivate the construction of $M_{i_{k}-2},$ recall that if one would already know $M_{i_{k}-2},$ then one would have a commutative diagram
$$
\xymatrix{
M_{i_{k}-2}\ar[r]^{\Pi} \ar[d] & M_{i_{k}-1} \ar[d] \\
\Lie_{i_{k}-2} \ar[r]^{\simeq} & \Lie_{i_{k}-1},
}
$$
where the lower horizontal arrow would be an isomorphism. It would follow that 

$$\ker(M_{i_{k}-2}\rightarrow \Lie_{i_{k}-2})\simeq \lbrace m\in M_{i_{k}-2}\colon \Pi(m)\in \ker(M_{i_{k}-1}\rightarrow \Lie_{i_{k}-1}) \rbrace.$$

 This is the reason, why I define
$$
I_{i_{k}-2}=\lbrace m \in M_{i_{k}-1}\colon \Pi(m)\in I_{i_{k}-1} \rbrace
$$
and $M_{i_{k}-2}=W\otimes_{F,W} I_{i_{k}-2},$ and $V\colon M_{i_{k}-2}\rightarrow I_{i_{k}-2}\subseteq M_{i_{k}-1}$ as before. To define the operator $\Pi,$ note that I can also write $I_{i_{k}-2}=\lbrace m\in M_{i_{k}-1}\colon U(m)\in M_{i_{k}-1} \rbrace,$ so that $\Pi=U\circ V\colon M_{i_{k}-2}\rightarrow M_{i_{k}-1}$ is well-defined. This algorithm permits one then to define the inverse morphism to $\Phi_{k}.$ In particular, it allows one to write down some explicit equations, as I explain now:
\\
One can for example look at the case, where only the index $i=d-1$ is critical. Then I can take for the $d-1$-th graded piece just the $W$-sublattice of $N_{d-1}$ generated by $e_{1},...,e_{d}$ and say $\Lie_{d-1}$ is generated by the image of $e_{1}.$ Then I find uniquely determined scalars $\xi_{2},...,\xi_{d}\in k,$ such that $\ker(P_{d-1}\rightarrow \Lie_{d-1})$ has the basis $pe_{1},e_{2}-[\xi_{2}]e_{1},...,e_{d}-[\xi_{d}]e_{1}$ and this is exactly $Q_{d-1}.$ Now $P_{d-2}$ is simply the image under $\dot{F},$ i.e. under applying the Wittvector-Frobenius, so that this $W$-module has a basis given by $f_{1}^{[d-2]}=pe_{1},f_{2}^{[d-2]}=e_{2}-[\xi_{2}^{p}]e_{1},...,f_{d}^{[d-2]}=e_{d}-[\xi_{d}^{p}]e_{1}.$ Since I assumed that $i=d-2$ is not critical, one can compute $Q_{d-2}$ as follows: If I denote by $\lambda_{i}^{[(d-2)]}\in k$ the uniquely determined scalar, such that $\lambda_{i}^{[(d-2)]}\cdot e_{1}=\text{Im}(f_{i}^{[(d-2)]})$ inside of $\Lie_{(d-1)},$ then I have that
$$
Q_{(d-2)}=\langle pe_{1},pe_{2},...,[\lambda_{j}^{[(d-2)]}]f_{2}^{[(d-2)]}-[\lambda_{2}^{(d-2)}]f_{j}^{[(d-2)]},... \rangle_{W}.
$$
Here one has of course that $\lambda_{i}^{[(d-2)]}=[\xi_{i}^{p}]-[\xi].$ Continuing this algorithm, one gets equations for the special $\mathcal{O}_{D}$-displays with only the index $i=d-1$ critical.
\section{The basic construction}
From now on I have to restrict to the case $d=2$ unfortunately. The first aim is to construct two special $\mathcal{O}_{D}$-displays $\mathcal{P}^{(0)}$ resp. $\mathcal{P}^{(1)}$ over $\Spec(k[X,\frac{1}{X-X^{p}}])$ resp. $\Spec(k[Y,\frac{1}{Y-Y^{p}}]),$ which are included in the standard Isodisplay I fixed:
$$
\xymatrix{
N_{0}\ar[r]^{\Pi=id} & N_{1} \ar[r]^{\Pi=\cdot p} & N_{0}.
}
$$
Recall that I have chosen a trivialization of $N,$ i.e. a $W_{\mathbb{Q}}(k)$-basis of $e_{1},e_{2}$ of $N_{0}.$ Then I define 

$P^{(0)}=P^{(0)}_{0}\oplus P^{(0)}_{1},$ $Q^{(0)}=Q^{(0)}_{0} \oplus Q^{(0)}_{1}$

 via
\begin{enumerate}
\item[$\bullet$] $P^{(0)}_{0}=\langle pe_{1}, e_{2}-[X^{p}]e_{1} \rangle_{W(k[X,\frac{1}{X-X^{p}}])} $
\item[$\bullet$] $P^{(0)}_{1}=\langle e_{1}, e_{2} \rangle_{W(k[X,\frac{1}{X-X^{p}}])} $
\item[$\bullet$] $Q^{(0)}_{0}=\langle pe_{1},pe_{2}  \rangle_{W(k[X,\frac{1}{X-X^{p}}])} $
\item[$\bullet$] $Q^{(0)}_{1}=\langle pe_{1},e_{2}-[X]e_{1} \rangle_{W(k[X,\frac{1}{X-X^{p}}])} $
\end{enumerate}
	
	With the following operators:

\begin{enumerate}
\item[$\bullet$] $\Pi\colon P^{(0)}_{0}\rightarrow P^{(0)}_{1}$ is induced by $$\Pi=\text{id}\colon N_{0}\otimes_{W_{\mathbb{Q}}(k)}W_{\mathbb{Q}}(k[X,\frac{1}{X-X^{p}}])\rightarrow N_{1}\otimes_{W_{\mathbb{Q}}(k)}W_{\mathbb{Q}}(k[X,\frac{1}{X-X^{p}}])$$ and $\Pi\colon P^{(0)}_{1}\rightarrow P^{(0)}_{0}$ is induced by

$$\Pi=\cdot p\colon N_{1}\otimes_{W_{\mathbb{Q}}(k)}W_{\mathbb{Q}}(k[X,\frac{1}{X-X^{p}}])\rightarrow N_{0}\otimes_{W_{\mathbb{Q}}(k)}W_{\mathbb{Q}}(k[X,\frac{1}{X-X^{p}}]).$$
\item[$\bullet$] $\dot{F}\colon Q^{(0)}_{0}\rightarrow P^{(0)}_{1}$ given by $F/p$ and $\dot{F}\colon Q^{(0)}_{1}\rightarrow P^{(0)}_{0}$ given by $F.$
\end{enumerate}
This construction is summarized in the following diagram:
$$
\xymatrix{
     & 0 & 1 & 0 \\
 P: & pe_{1}, e_{2}-[X^{p}]e_{1} & e_{1},e_{2} & pe_{1},e_{2}-[X^{p}]e_{1} \\
 Q: & pe_{1},pe_{2} & pe_{1},e_{2}-[X]e_{1} & pe_{1},pe_{2}.
}
$$
Note, that one has then the following commutative diagram
$$
\xymatrix{
P^{(0)}_{0} \ar[r]^{\Pi} \ar[d]& P^{(0)}_{1} \ar[r]^{\Pi} \ar[d] & P^{(0)}_{0} \ar[d] \\
\Lie(P^{(0)})_{0} \ar[r]^{0} & \Lie( P^{(0)}_{1}) \ar[r]^{X-X^{p}} & \Lie(P^{(0)}_{0}).
}
$$
Next, I define the second family $\mathcal{P}^{(1)}:$  I define $P^{(1)}=P^{(1)}_{0}\oplus P^{(1)}_{1},$ $Q^{(1)}=Q^{1)}_{0} \oplus Q^{(1)}_{1}$ via
\begin{enumerate}
\item[$\bullet$] $P^{(1)}_{0}=\langle pe_{1}, e_{2} \rangle_{W(k[Y,\frac{1}{Y-Y^{p}}])} $
\item[$\bullet$] $P^{(1)}_{1}=\langle e_{1}, e_{1}-[Y^{p}]e_{2}/p \rangle_{W(k[Y,\frac{1}{Y-Y^{p}}])} $
\item[$\bullet$] $Q^{(1)}_{0}=\langle pe_{1}-[Y]e_{2},pe_{2}  \rangle_{W(k[Y,\frac{1}{Y-Y^{p}}])} $
\item[$\bullet$] $Q^{(1)}_{1}=\langle pe_{1}, e_{2} \rangle_{W(k[Y,\frac{1}{Y-Y^{p}}])}.$
\end{enumerate}
Here I take the same operators as before. Again, one can summarize this construction in the following diagram:
$$
\xymatrix{
     & 0 & 1 & 0 \\
 P: & pe_{1}, e_{2} & e_{1}-[Y^{p}]e_{2}/p,e_{2} & pe_{1},e_{2}\\
 Q: & pe_{1}-[Y]e_{2},pe_{2} & pe_{1},e_{2} & pe_{1}-[Y]e_{2},pe_{2}.
}
$$
Note that one has the following commutative diagram
$$
\xymatrix{
P^{(1)}_{0} \ar[r]^{\Pi} \ar[d] & P^{(1)}_{1} \ar[r]^{\Pi} \ar[d] & P^{(1)}_{0} \ar[d] \\
\Lie(P^{(1)}_{0}) \ar[r]^{Y-Y^{p}} & \Lie(P^{(1)}_{1}) \ar[r]^{0} & \Lie(P^{(1)}_{0}).
}
$$
The next step is now to glue these families $\mathcal{P}^{(0)}$ and $\mathcal{P}^{(1)}$ to one family $\mathcal{P}^{(0,1)}$ over $$\Spec(k[X,Y,\frac{1}{1-X^{p-1}},\frac{1}{1-Y^{p-1}}]/(X\cdot Y)),$$ in such a way that if one restricts to the open $D(X),$ one gets back the family $\mathcal{P}^{(0)}$ and if one restricts to the open $D(Y)$ one gets back the family $\mathcal{P}^{(1)}.$ This glueing will be done using Ferrand-glueing.
\\
For this the key observation is that $\mathcal{P}^{(0)}$ resp. $\mathcal{P}^{(1)}$ are defined also when $X=0$ resp. $Y=0$ and then they are indeed just identical. 
\\
Note first, that the ring $k[X,Y,\frac{1}{1-X^{p-1}},\frac{1}{1-Y^{p-1}}]/(X\cdot Y)$ identifies with the following fiber product\footnote{Indeed, the fiber product is given by pairs of polynomials $(f_{1},f_{2})\in k[X,\frac{1}{1-X^{p-1}}]\times k[Y,\frac{1}{1-Y^{p-1}}],$ whose constant coefficient agree. This defines an element of $k[X,Y,\frac{1}{1-X^{p-1}},\frac{1}{1-Y^{p-1}}]/(X\cdot Y).$ Conversely, an element $P\in k[X,Y,\frac{1}{1-X^{p-1}},\frac{1}{1-Y^{p-1}}]/(X\cdot Y)$ has no mixed terms in $X$ and $Y,$ so one can uniquely write it as $a_{0}+f_{1}(X)+f_{2}(X),$ where $(f_{1},f_{2})\in k[X,\frac{1}{1-X^{p-1}}]\times k[Y,\frac{1}{1-Y^{p-1}}]$ have no constant term. This gives the inverse.}
$$
\xymatrix{
k[X,\frac{1}{1-X^{p-1}}]\times_{k} k[Y,\frac{1}{1-Y^{p-1}}] \ar[r] \ar[d] & k[Y,\frac{1}{1-Y^{p-1}}] \ar[d]^{Y\mapsto 0} \\
k[X,\frac{1}{1-X^{p-1}}] \ar[r]^{X\mapsto 0} & k.
}
$$
Observe that the functor of taking $p$-typical Wittvectors commutes with fiberproducts, so that one gets the cartesian diagram
$$
\xymatrix{
W(k[X,Y,\frac{1}{1-X^{p-1}},\frac{1}{1-Y^{p-1}}]/(X\cdot Y)) \ar[r]^-{i_{1}} \ar[d]^{i_{2}} & W(k[Y,\frac{1}{1-Y^{p-1}}]) \ar[d]^{j_{1}} \\
W(k[X,\frac{1}{1-X^{p-1}}]) \ar[r]^{j_{2}} & W(k).
}
$$
This will allow one to construct a special $\mathcal{O}_{D}$-display $(\mathcal{P}^{(0,1)},\Pi, \mathbb{Z}/d\mathbb{Z}-\text{grad}),$ which admits an inclusion, that is compatible with all extra structure, towards $$N\otimes_{W_{\mathbb{Q}}(k)} W_{\mathbb{Q}}(k[X,Y,\frac{1}{1-X^{p-1}},\frac{1}{1-Y^{p-1}}]/(X\cdot Y)).$$ It will have the crucial property, that now one can set either the zeroth or the first index critical and such that over $D(X),$ one gets back the family $\mathcal{P}^{(0)}$ and over $D(Y),$ one gets back the family $\mathcal{P}^{(1)}.$ 
The idea is to glue the two families $(\mathcal{P}^{(0)},\Pi,\mathbb{Z}/d\mathbb{Z}-\text{grad},\rho^{(0)})$ and $(\mathcal{P}^{(1)},\Pi,\mathbb{Z}/d\mathbb{Z}-\text{grad},\rho^{(1)})$ along the closed point of intersection, which is where both indices are critical, i.e. where $X=0$ resp. $Y=0.$ 
\\
First note, that both families are also defined over $k[X,\frac{1}{1-X^{p-1}}]$ resp. $k[Y,\frac{1}{1-Y^{p-1}}].$ One then has an isomorphism of special $\mathcal{O}_{D}$-displays
\begin{equation}
\text{id}\colon \mathcal{P}^{(1)}\otimes_{W(k[Y,\frac{1}{1-Y^{p-1}}]),j_{2}}W(k)\simeq \mathcal{P}_{0} \simeq \mathcal{P}^{(0)}\otimes_{W(k[X,\frac{1}{1-X^{p-1}}]),j_{1}}W(k).
\end{equation}
Let me write in the following $\overline{A}_{\Delta^{\text{std}}}=k[X,Y,\frac{1}{1-X^{p-1}},\frac{1}{1-Y^{p-1}}]/(X\cdot Y).$
\\
 I will construct a nilpotent display $\mathcal{P}^{(0,1)}=(P^{(0,1)},Q^{(0,1)},\dot{F},F)$ as follows: Let $P^{(0,1)}\subset P^{(0)}\times P^{(1)}$ be the abelian subgroup given by pairs $(p_{0},p_{1}),$ such that $j_{1}(p_{0})=j_{2}(p_{1}).$ I give it the structure of a $W(\overline{A}_{\Delta^{\text{std}}})$-module by
$$
a\cdot (p_{0},p_{1})=(i_{1}(a)\cdot p_{0},i_{2}(a)\cdot p_{1}).
$$
Next define $L^{(0,1)}\subset P^{(0)}/Q^{(0)}\times P^{(1)}/Q^{(1)}$ similiarly as pairs agreeing upon applications of $j_{1}$ resp. $j_{2}.$ Then $P^{(0,1)}$ resp. $L^{(0,1)}$ are free $W(\overline{A}_{\Delta^{\text{std}}})$ resp. $\overline{A}_{\Delta^{\text{std}}}$-modules of rank $4$ resp. $2.$ Then define the free $W(\overline{A}_{\Delta^{\text{std}}})$-module $Q^{(0,1)}\subset P^{(0,1)}$ by
$$
\ker(P^{(0,1)}\rightarrow L^{(0,1)}).
$$
The operators $\dot{F}^{(0)}, \dot{F}^{(1)}$ resp. $F^{(0)},F^{(1)}$ induce uniquely determined operators $\dot{F}^{(0,1)}\colon Q^{(0,1)}\rightarrow P^{(0,1)}$ and $F\colon P^{(0,1)}\rightarrow P^{(0,1)}.$ This completes the construction of the display $\mathcal{P}^{(0,1)}.$
\\
The extra structure on $\mathcal{P}^{(0)}$ resp. $\mathcal{P}^{(1)}$ induces the respective extra structure on $\mathcal{P}^{(0,1)},$ so that I get in total a special $\mathcal{O}_{D}$-display $(\mathcal{P}^{(0,1)},\Pi,\mathbb{Z}/d\mathbb{Z}-\text{grad},\rho).$ 
Note that one has the following commutative diagram
$$
\xymatrix{
P^{(0,1)}_{0} \ar[r]^{\Pi} \ar[d] & P^{(0,1)}_{1} \ar[r]^{\Pi} \ar[d] & P^{(0,1)}_{0} \ar[d] \\
\Lie(P^{(0,1)}_{0}) \ar[r]^{Y-Y^{p}} & \Lie(P^{(0,1)}_{1}) \ar[r]^{X-X^{p}} & \Lie(P^{(1)}_{0}).
}
$$
It remains to lift this construction to all $p$-adic thickenings. To do this, one has to give a lift of the Hodge-Filtration of $(\mathcal{P}^{(0,1)},\Pi, \mathbb{Z}/d\mathbb{Z}-\text{grad}).$ In order to prepare the use of Grothendieck-Messing deformation theory, I will perform the variabel transformation $Y-Y^{p}\mapsto Y^{\prime},$ $X-X^{p}\mapsto X^{\prime}$ (this is an isomorphism, since I inverted $1-X^{p-1}$ resp. $1-Y^{p-1}$).
\\
I consider the pd-thickening
$$
A_{\Delta^{\text{std}},n}\rightarrow A_{\Delta^{\text{std}},0},
$$
where $A_{\Delta^{\text{std}},n}=\breve{\mathbb{Z}}_{p}[X,Y,\frac{1}{1-X^{p-1}},\frac{1}{1-Y^{p-1}}]/(X\cdot Y -p, p^{n})$ and $A_{\Delta^{\text{std}},0}=k[X,Y,\frac{1}{1-X^{p-1}},\frac{1}{1-Y^{p-1}}]/(X\cdot Y).$

Then one finds a uniquely determined special $\mathcal{O}_{D}$-$\mathcal{W}(A_{\Delta^{\text{std}},n}/A_{\Delta^{\text{std}},0})$-window $(\tilde{\mathcal{P}^{(0,1)}},\Pi,\mathbb{Z}/d\mathbb{Z}-\text{grad}),$ which lifts the just constructed special $\mathcal{O}_{D}$-Display over $A_{\Delta,0}$ by Lemma \ref{Lemma zu relativen SFD Windows}. One finds a commutative diagram
$$
\xymatrix{
\tilde{\mathcal{P}^{(0,1)}}_{0}/I(A_{\Delta^{\text{std}},n}) \ar[d] \ar[r]^{\Pi} & \tilde{\mathcal{P}^{(0,1)}}_{1}/I(A_{\Delta^{\text{std}},n}) \ar[d] \ar[r]^{\Pi} & \tilde{\mathcal{P}^{(0,1)}}_{0}/I(A_{\Delta^{\text{std}},n}) \ar[d] \\
A_{\Delta^{\text{std}},n} \ar[r]^{\cdot X} & A_{\Delta^{\text{std}},n} \ar[r]^{\cdot Y} & A_{\Delta^{\text{\text{std}}},n},
}
$$
which after the variabel transformation one performed previously lifts the corresponding diagram relating the Hodge-Filtrations for $(\mathcal{P}^{(0,1)},\Pi, \mathbb{Z}/d\mathbb{Z}-\text{grad}).$ By the above Lemma \ref{Deformationstheorie von SFD-displays}, this suffices to construct a lift $(\mathcal{P}^{(0,1)}_{n},\Pi,\mathbb{Z}/d\mathbb{Z}-\text{grad})$ of $(\mathcal{P}^{(0,1)},\Pi, \mathbb{Z}/d\mathbb{Z}-\text{grad})$ to a special $\mathcal{O}_{D}$-display over $A_{\Delta^{\text{std}},n}.$ Since the quasi-isogeny lifts uniquely along nilpotent thickenings, one therefore has constructed a point $$(\mathcal{P}^{(0,1)}_{n},\Pi,\mathbb{Z}/d\mathbb{Z}-\text{grad},\rho_{n})\in \mathcal{D}^{(0)}(\Spec(A_{\Delta^{\text{std}},n})).$$
The reason I wrote before $\Delta^{\text{std}}$ is that this is the standard simplex in $\mathcal{B}T(\PGL_{2}(\mathbb{Q}_{p})):$ It is given by $\Lambda^{\text{std}}=\langle e_{1},e_{2}\rangle_{\mathbb{Z}_{p}},$ $\Lambda^{\prime \text{std}}=\langle pe_{1},e_{2} \rangle_{\mathbb{Z}_{p}}.$ Then I have that 
$$
\widehat{U}_{\Delta^{\text{std}}}=\Spf(A_{\Delta^{\text{std}}})
$$
(in general this isomorphism is non canonical, since it depends on the choice of a basis).
Using that the group $\GL_{2}(\mathbb{Q}_{p})$ acts transitively on maximal simplices in the building, one may then use the group-action to extend the construction of the special formal $\mathcal{O}_{D}$-Display, with quasi-isogeny, to all simplices. By construction, these displays agree on overlaps and one therefore has constructed a morphism
$$
\Upsilon\colon \widehat{\Omega}^{2}_{\breve{\mathbb{Z}_{p}}}\rightarrow \mathcal{D},
$$
which one forced to be equivariant for the $\GL_{2}(\mathbb{Q}_{p})$-action.
\section{Reductions}
I first start with a reduction lemma, which says the following: To show that the natural transformation
$$
\Upsilon\colon \widehat{\Omega}^{2}_{\breve{\mathbb{Z}}_{p}}\rightarrow \mathcal{D}
$$
is an isomorphism, it is enough to do so, when restricted to test-objects $\Spec(R)\in \Nilp_{\breve{\mathbb{Z}}_{p}},$ such that $pR=0.$ For this, let me fix the following notation: Let $\AlgNilp_{\breve{\mathbb{Z}}_{p},n}$ be the category of $\breve{\mathbb{Z}}_{p}$-algebras $R,$ such that $p^{n}R=0$ and let $\Upsilon_{n}$ be the restriction of $\Upsilon$ to these test-objects.
\begin{Lemma}
Let $n\geq 1$ be some integers and assume that $\Upsilon_{n}$ is an isomorphism, then also $\Upsilon_{n+1}$ is one.
\end{Lemma}
\begin{proof}
I will follow the argument of Drinfeld (\cite[Prop. 2.5.]{DrinfeldOmega}); though of course here the natural transformation goes in the other direction. One first shows the following: let $\Delta$ be some simplex in the Bruhat-Tits tree for $\PGL_{2}(\mathbb{Q}_{p})$ and consider $B\in \AlgNilp_{\breve{\mathbb{Z}}_{p},n+1}$ and let $B^{\prime}=B/p^{n}B$ and $f\colon B\rightarrow B^{\prime}$ the quotient morphism. Fix a point $\alpha\in \widehat{U}_{\Delta}(\Spec(B^{\prime})),$ which corresponds to a pair $(\alpha_{1},\alpha_{2})\in B^{\prime},$ such that $\alpha_{1}\cdot \alpha_{2}=p$ and $1-\alpha_{1}^{p-1},1-\alpha_{2}^{p-1}\in (B^{\prime})^{*}.$ Let $\Upsilon_{\Delta}(\alpha)\in \mathcal{D}(\Spec(B^{\prime}))$ be the image of $\alpha.$ Then consider
$$
\widehat{U}_{\Delta f,\alpha}(B)=\lbrace \beta\in \widehat{U}_{\Delta}(\Spec(B))\colon f(\beta)=\alpha \rbrace
$$
and
$$
\mathcal{D}_{f,\Upsilon_{\Delta}(\alpha)}(B)=\lbrace y\in \mathcal{D}(\Spec(B))\colon f(y)=\Upsilon_{\Delta}(\alpha) \rbrace,
$$
so that one gets an induced morphism $\Upsilon_{\Delta}\colon \widehat{U}_{\Delta f,\alpha}(B)\rightarrow \mathcal{D}_{f,\Upsilon_{\Delta}(\alpha)}(B).$ The key point to verify is the following: if $\mathcal{D}_{f,\Upsilon_{\Delta}(\alpha)}(B)\neq \emptyset,$ then also $\widehat{U}_{\Delta f,\alpha}(B)\neq \emptyset.$
\\
Let $y\in \mathcal{D}_{f,\Upsilon_{\Delta}(\alpha)}(B),$ which corresponds to a pair $((\mathcal{P},\Pi,\mathbb{Z}/d\mathbb{Z}-\text{grad}),\rho)$ over $B.$ The quasi-isogeny always deforms uniquely along nilpotent thickening. Let me denote by $f^{*}((\mathcal{P},\Pi,\mathbb{Z}/d\mathbb{Z}-\text{grad}))$ the base-change along $f\colon B\rightarrow B^{\prime}.$ Then by assumption, there is an isomorphism $f^{*}((\mathcal{P},\Pi,\mathbb{Z}/d\mathbb{Z}-\text{grad}))\simeq \mathcal{P}_{\Upsilon(\alpha)}.$ Recall that all graded-pieces of $P_{\Upsilon(\alpha)}/Q_{\Upsilon(\alpha)}$ are by construction free $B^{\prime}$-modules of rank $1.$ It follows therefore, that $P_{i}/Q_{i}\otimes_{B} B^{\prime}\simeq P_{i,\Upsilon(\alpha)}/Q_{i,\Upsilon(\alpha)}$ are also free of rank $1.$ Now, since $P_{i}/Q_{i}$ is finite-projective over $B,$ so that $\text{Tor}^{1}_{B}(B^{\prime},P_{i}/Q_{i})=0$ and the ideal $p^{n}B$ is nilpotent, it follows that also $P_{i}/Q_{i}$ is free of rank $1$ (c.f. \cite[Tag 051H]{stacks}). I consider the diagram
$$
\xymatrix{
P_{0}/I(B)P_{0} \ar[r]^{\Pi} \ar[d] & P_{1}/I(B)P_{1} \ar[r]^{\Pi} \ar[d] & P_{0}/I(B)P_{1} \ar[d] \\
P_{0}/Q_{0} \ar[r]^{\cdot c_{1}} & P_{1}/Q_{1} \ar[r]^{\cdot c_{2}} & P_{0}/Q_{0},
}
$$
where $c_{1},c_{2}\in B$ and the commutativity of this diagram says that $c_{1}\cdot c_{2}=p.$ Again, by the assumption that $f^{*}((\mathcal{P},\Pi,\mathbb{Z}/d\mathbb{Z}-\text{grad}))\simeq \mathcal{P}_{\Upsilon(\alpha)}$ and the fact that in the commutative diagram for $\mathcal{P}_{\Upsilon(\alpha)}$ as the one I used just now, the lower horizontal maps are given by multiplication with $\alpha_{1}$ resp. $\alpha_{2},$ one sees that $f(c_{1})=\alpha_{1}$ and $f(c_{2})=\alpha_{2}.$ This implies in particular, that $1-c_{1}^{p-1}$ and $1-c_{2}^{p-1}$ are invertible in $B$ (since this is a condition one may check modulo a nilpotent ideal). In total, one constructed the desired morphism
$$
\Spec(B)\rightarrow \widehat{U}_{\Delta, f, \alpha}.
$$
Now, since $\Upsilon$ is constructed via glueing of all the $\Upsilon_{\Delta},$ one also gets (with the obvious meaning of this notation) the following: If $\mathcal{D}_{f,\Upsilon(\alpha)}(\Spec(B))\neq \emptyset,$ then also $\widehat{\Omega}^{2}_{f,\alpha}(\Spec(B))\neq \emptyset.$
\\
From here, one may run the same argument as Drinfeld does: It is enough to show that $\Upsilon$ induces an isomorphism between $\widehat{\Omega}^{2}_{f,\alpha}(\Spec(B))$ and $\mathcal{D}_{f,\Upsilon(\alpha)}(\Spec(B)).$ By deformation theory one sees that if $\widehat{\Omega}^{2}_{f,\alpha}(\Spec(B))\neq \emptyset$ resp. $\mathcal{D}_{f,\Upsilon(\alpha)}(\Spec(B))\neq \emptyset,$ then these are principal homogenouse spaces under the groups $\widehat{\Omega}^{2}_{f,\alpha}(B^{\prime}[I])$ resp. $\mathcal{D}_{f,\Upsilon(\alpha)}(B^{\prime}[I]),$ where $I=p^{n}B, B^{\prime}[I]=B^{\prime}\oplus I,$ with $I^{2}=0.$ From the assumption that $\Upsilon_{n}$ is an isomorphism, one deduces then that both groups are isomorphic and by the previous step one can then conclude that $\Upsilon$ induces an isomorphism between $\widehat{\Omega}^{2}_{f,\alpha}(\Spec(B))$ and $\mathcal{D}_{f,\Upsilon(\alpha)}(\Spec(B)),$ as desired.
\end{proof}
One is therefore left with showing that
$$
\Upsilon_{0}\colon \widehat{\Omega}^{2}_{\breve{\mathbb{Z}}_{p}}\times_{\Spf(\breve{\mathbb{Z}}_{p})}\Spec(\overline{\mathbb{F}}_{p})\rightarrow \mathcal{D}\times_{\Spf(\breve{\mathbb{Z}}_{p})}\Spec(\overline{\mathbb{F}}_{p})
$$
is an isomorphism. Here, I will use the following, almost trivial, statement:
\begin{Lemma}
Let $k$ be an algebraically closed field and $f\colon X\rightarrow Y$ be a morphism between projective varieties over $\Spec(k).$ Assume that $f$ is a bijection on $\Spec(k)$-valued points and in addition étale, then $f$ is an isomorphism.
\end{Lemma}
\begin{proof}
Let $s\colon X\rightarrow \Spec(k)$ and $t\colon Y\rightarrow \Spec(k)$ be the two structure-morphisms. Since $t$ is seperated and $s=t\circ f$ is proper, also $f$ is proper. Now an étale morphism is in particular quasi-finite and quasi-finite plus proper implies finite. Since $k$ is algebraically closed and $f(k)$ is a bijection, it follows that $f$ is finite étale of degree $1,$ i.e. an isomorphism.
\end{proof}
\begin{Lemma}
The functor $\mathcal{D}\times_{\Spf(\breve{\mathbb{Z}}_{p})}\Spec(\overline{\mathbb{F}}_{p})$ can be written as $\colim_{n,m\in \mathbb{Z}}\overline{\mathcal{M}}_{n,m},$ where $\mathcal{D}_{0, n,m}$ are subfunctors, which are representable by projective $\overline{\mathbb{F}}_{p}$-schemes.
\end{Lemma}
\begin{proof}
If I take a point $(X,\rho)\in \mathcal{D}_{0}(\Spec(R)),$ $R$ a $k$-algebra, then I find some $n,$ such that $p^{n}\rho\colon \mathbb{X}_{R}\rightarrow X$ is a $\mathcal{O}_{D}$-linear isogeny. Then let $\mathcal{D}_{0,n,m}$ be the subfunctor parametrizing pairs $(X,\rho),$ such that $p^{n}\rho$ is a $\mathcal{O}_{D}$-linear isogeny of height $m.$ Then it follows that $\mathcal{D}_{0}=\colim_{n,m}\mathcal{D}_{0,n,m}$ and using the representability of the Hilbert-scheme it is easy to see that $\mathcal{D}_{0,n,m}$ are representable by projective schemes. 
\end{proof}
Now let $((\mathcal{P}_{\widehat{\Omega}},\Pi,\mathbb{Z}/d\mathbb{Z}-\text{grad}),\rho_{\widehat{\Omega}})$ be the object one constructed before over $ \widehat{\Omega}^{2}_{\breve{\mathbb{Z}}_{p}}.$ Then one can look at the sublocus of $\Omega^{2}_{k},$ where $p^{n}\rho_{\Omega}$ is a $\mathcal{O}_{D}$-linear isogeny of height $m;$ let me denote it by $ \Omega^{2}_{k,n,m}.$ This a projective $k$-scheme which is locally closed inside the whole $\Omega^{2}_{k}$, such that $\Omega^{2}_{k}=\colim_{n,m}\Omega^{2}_{k,n,m}$ and the natural transformation respects this union, i.e. one has
$$
\Upsilon_{0}=\colim_{n,m}\Upsilon_{0,n,m},
$$
where $\Upsilon_{0,n,m}\colon \Omega^{2}_{k,n,m}\rightarrow \mathcal{D}_{0,n,m}$ is the induced morphism, to which one is then able to apply the above Lemma.
In total, after these reductions one is left with showing that $\Upsilon(k)$ is a bijection and that $\Upsilon_{0}$ is étale.
\section{End of the proof}
\subsection{The special fiber}
\begin{Lemma}
The map $\Upsilon(k)$ is a bijection.
\end{Lemma}
\begin{proof}
This property was insured by construction.
\end{proof}
\subsection{Deformationtheory}
The aim of this section is to show the following
\begin{Lemma}\label{Lemma Upsilon ist etale}
The morphism $\Upsilon_{0}$ is étale.
\end{Lemma}
To this end, I will first study lifts of a point $x=(\mathcal{P},\Pi,\mathbb{Z}/2\mathbb{Z}-\text{grad},\rho)\in \mathcal{D}(\Spec(k))$ towards $k[\epsilon],$ $\epsilon^{2}=0.$ By rigidity of quasi-isogeny, this means that I have to understand possible lifts of the $\mathcal{O}_{D}$-display $(\mathcal{P},\Pi,\mathbb{Z}/2\mathbb{Z}-\text{grad}).$ The next statement is easy to check:
\begin{Lemma}\label{Lemma zu Tangentialraumen im Drinfeld Modulproblem}
The tangent space $T_{\mathcal{D},x}$ of $\mathcal{D}$ at the point $x$ is isomorphic to the $k$-vector space
$$
\lbrace u\in \Hom_{k}(Q/pP,(\epsilon)\otimes_{k}P/Q)\colon \deg(u)=0, u\Pi=\Pi u \rbrace.
$$
\end{Lemma}
Now the $k$-dimension of this tangent space depends on how many indices of $(\mathcal{P},\Pi,\mathbb{Z}/2\mathbb{Z}-\text{grad})$ are critical.
\begin{Lemma}
\begin{enumerate}
\item[(a):] If only one index is critical for $(\mathcal{P},\Pi,\mathbb{Z}/2\mathbb{Z}-\grad),$ then $\dim_{k}(T_{\mathcal{D},x})=1.$
\item[(b):] If both indices are critical for $(\mathcal{P},\Pi,\mathbb{Z}/2\mathbb{Z}-\grad),$ then $\dim_{k}(T_{\mathcal{D},x})=2.$
\end{enumerate}
\end{Lemma} 
\begin{proof}
First observe that the index $i=0$ is critical if and only if $\Pi\colon Q_{1}/pP_{1}\rightarrow Q_{0}/pP_{0}$ is the zero morphism and the index $i=1$ is critical if and only if $\Pi\colon Q_{0}/pP_{0}\rightarrow Q_{1}/pP_{1}$ is the zero morphism (use that $i$ is critical if and only if $\Pi P_{i}=Q_{i+1}$). It follows that the equation $\Pi u_{i}=u_{i+1} \Pi$ can only be satisfied in the case where $i$ is critical, but not $i+1,$ when $u_{i+1}=0.$ In case both indices are critical, these identities are always satisfied.
\end{proof}
Now one can turn to the proof of Lemma \ref{Lemma Upsilon ist etale}.
\begin{proof}
Let $x\in \widehat{\Omega}^{2}_{\breve{\mathbb{Z}_{p}}}(\Spec(k))$ and $\Upsilon(x)\in \mathcal{D}(k)$ the image. I will write 
$$
\Upsilon(x)=(\mathcal{P}_{0},\Pi,\mathbb{Z}/2\mathbb{Z}-\text{grad},\rho_{0}).
$$ Then I have to show that the natural transformation $\Upsilon$ induces an isomorphism on tangentspaces:
$$
d\Upsilon\colon T_{\widehat{\Omega}^{2}_{\breve{\mathbb{Z}}_{p}},x}\rightarrow T_{\mathcal{D},\Upsilon(x)}.
$$
This morphism is explicitly given as follows: Let $\Delta$ be some simplex, such that $x\in \overline{U}_{\Delta}.$ Then one has that 
$$
T_{\widehat{\Omega}^{2}_{\breve{\mathbb{Z}}_{p}},x}=T_{\bar{U}_{\Delta},x}=\lbrace \alpha=(\alpha_{1},\alpha_{2})\in k[\varepsilon]^{2}\colon 1-\alpha_{i}^{p-1}\in (k[\varepsilon])^{*}, \alpha_{1}\cdot \alpha_{2}=0, \alpha \equiv x (\text{mod})(\varepsilon) \rbrace,
$$
in the case where $\Delta$ is a maximal simplex. Then the image $d\Upsilon(\alpha)$ is the deformation $\mathcal{P}^{\prime}$ of $\mathcal{P}_{x},$ whose Hodge-Filtration is described by the following commutative diagram:
$$
\xymatrix{
P_{0}^{\prime} \ar[r]^{\Pi} \ar[d] & P_{1}^{\prime} \ar[r]^{\Pi} \ar[d] & P_{0}^{\prime} \ar[d] \\
\Lie(P_{0}^{\prime}) \ar[r]^{\cdot \alpha_{1}} & \Lie(P_{1}^{\prime}) \ar[r]^{\cdot \alpha_{2}} & \Lie(P_{0}^{\prime}).
}
$$
I will first make the following observations:
\begin{enumerate}
\item[(a):] Recall that $(\mathcal{P}_{0},\Pi,\mathbb{Z}/2\mathbb{Z}-\text{grad},\rho_{0})$ was the image $\Upsilon(x).$ Assume that the index $i$ is critical for $\mathcal{P}_{0}$ and let $\mathcal{P}^{\prime}$ be some deformation of $\mathcal{P}_{0}$ to $k[\varepsilon].$ Via Lemma \ref{Lemma zu Tangentialraumen im Drinfeld Modulproblem} it corresponds to a pair $u=(u_{1},u_{2}).$ Then the index $i$ is critical for $\mathcal{P}^{\prime}$ if and only if $u_{i+1}=0.$
\item[(b):] Let $\alpha=(\alpha_{1},\alpha_{2})\in T_{\widehat{\Omega}^{2}_{\breve{\mathbb{Z}}_{p}},x}$ and $\mathcal{P}^{\prime}$ be the deformation $d\Upsilon(\alpha)$ of $\mathcal{P}.$ Assume that $i$ is critical for $\mathcal{P}^{\prime}$ (so that $\alpha_{i}=0$). Then $i+1$ is critical if and only if $u_{i}=0$ if and only if $\alpha_{i+1}=0.$ 
\end{enumerate}
Now I start showing that $d\Upsilon$ is an isomorphism. Let me first look at the case where $x$ is not a $\mathbb{F}_{p}$-rational point, which implies that $\Upsilon(x)=(\mathcal{P},\Pi,\mathbb{Z}/d\mathbb{Z}-\text{grad},\rho)$ is such that exactly one index is critical. It follows then that both tangent spaces $T_{\widehat{\Omega}^{2}_{\breve{\mathbb{Z}}_{p}},x}$ and $T_{\mathcal{D},\Upsilon(x)}$ are one-dimensional. Thus it suffices to show that $d\Upsilon$ is not the zero-morphism; this will be a direct calculation:
\\
Let $x$ correspond to the map $\breve{\mathbb{Z}}_{p}[X] \rightarrow k, S\mapsto \zeta$. Then a point of the tangent space to $x$ corresponds to the map $\breve{\mathbb{Z}}_{p}[X] \rightarrow k, S\mapsto \zeta + \epsilon\varrho$. The associated display $\mathcal{P}_{1}$ over $k[\epsilon]$ is given by the structural matrix
$$ \begin{bmatrix}
0 & 0 & [\zeta^{q^{2}}] - [\zeta - \epsilon\varrho]  & 1 \\
0 & 0 & 1 & 0 \\
0 & 1 & 0 & 0 \\
1 & 0 & 0 & 0
\end{bmatrix} $$
Denote by $\mathcal{P}_{0}^{\prime}$ the trivial lift of $\mathcal{P}_{0}$ given by the inclusion $k \rightarrow k[\epsilon]$. One has to calculate the element $u\in T_{\mathcal{P}_{0}}$, such that $\mathcal{P}_{1} = \mathcal{P}_{0}^{\prime} + u$.

Every lift of $\mathcal{P}_{0}$ is of the form $\mathcal{P}_{\alpha}$ for some $\alpha\colon P_{0}\rightarrow \epsilon P_{0}$ (for the definition of these displays, please see \cite[Example 22, pg. 21-22]{ZinkDisplay}). To calculate $u$ one calculates $\alpha$ for the lift $\mathcal{P}_{1}$. One has that $(P_{1},Q_{1})=(P_{0},Q_{0})$, so I just write $P$ and $Q$ instead.

Define $^{F}$-linear maps
$$ g\colon P \longrightarrow W(\epsilon) \otimes_{W(k[\epsilon])} P$$
and
$$ h\colon Q \longrightarrow W(\epsilon) \otimes_{W(k[\epsilon])} P$$
by 
$$ F_{1}(x) = F_{0}(x) - g(x), x \in P,$$
$$ \dot{F}_{1}(y) = \dot{F}_{0}(y) - h(y), y\in Q.$$
Then $\alpha\colon P_{0}\rightarrow \epsilon P_{0}$ is given by the equations $\alpha(F_{0}(x))=g(x)$ and $\alpha(\dot{F}_{0}(y))=h(y)$.
A calculation shows that $\alpha(b_{j})=0$ for all $j\neq 3$ and $\alpha(b_{3})=\epsilon\varrho b_{1}$. 

To determine $u$ I first note that there is a canonical isomorphism $\epsilon P \cong \epsilon \otimes_{k_0} P/\pi P$, then $\alpha$ factorizes through 
$$ \bar{\alpha}: P/\pi P \longrightarrow \epsilon \otimes_{k} P/\pi P$$
and I have
$$ u\colon Q/\pi P \subset P/\pi P\longrightarrow \epsilon \otimes_{k} P/\pi P \longrightarrow \epsilon \otimes_{k} P/Q .$$ Therefore $u\neq 0$.
\\
Next, I consider the case, when $x$ is $\mathbb{F}_{p}$-rational. Then $\Upsilon(x)=(\mathcal{P}_{0},\Pi,\mathbb{Z}/d\mathbb{Z}-\text{grad},\rho)$ is such that both indices are critical and both tangent spaces are $2$-dimensional. The tangentspace $T_{\widehat{\Omega}^{2}_{\breve{\mathbb{Z}}_{p}},x}$ has two $1$-dimensional subspaces given by the conditions $\alpha_{i}=0,$ $i=0,1.$ If one restricts $d\Upsilon$ to one of these $1$-dimensional subspaces, I get an induced morphism to the space of deformations of $\mathcal{P},$ such that $i=0$ resp. $i=1$ is critical for that deformation. By observation (a) above these subspaces of $T_{\mathcal{D},\Upsilon(x)}$ are one-dimensional and by observation (b) this induced morphism is injective, thus an isomorphism. The images of these induced morphisms are distinct in $T_{\mathcal{D},\Upsilon(x)},$ so that $d\Upsilon$ has to be an isomorphism.
\end{proof}


\newpage
\bibliography{/Users/sebastianbartling/Documents/Bibtex/mybib}{}
\bibliographystyle{alpha}
\end{document}